\documentclass{amsart}

\usepackage{amsmath, amssymb, amsthm}
\usepackage{enumerate, float, pict2e, bxeepic, caption, subcaption, hyperref}

\DeclareMathOperator{\GL}{GL}
\DeclareMathOperator{\PG}{PG}
\DeclareMathOperator{\Aut}{Aut}
\DeclareMathOperator{\ind}{ind}

\def\D{{\mathcal D}}
\def\F{{\mathcal F}}
\def\P{{\mathcal P}}

\def\fano#1#2#3#4#5#6#7{
    \footnotesize
    \setlength{\unitlength}{0.0003in}
    \begin{picture}(4066,4500)(0,-250)
      \put(2033,1283){\ellipse{2400}{2400}}
      \put(83,83){\circle*{200}}
      \put(2033,83){\circle*{200}}
      \put(3983,83){\circle*{200}}
      \put(983,1883){\circle*{200}}
      \put(2033,1253){\circle*{200}}
      \put(3083,1883){\circle*{200}}
      \put(2033,3983){\circle*{200}}
      \path(2033,3983)(3983,83)(83,83)(2033,3983)
      \path(3983,83)(983,1883)
      \path(2033,3983)(2033,83)
      \path(83,83)(3083,1883)
      \put(-150,-300){$#1$}
      \put(4050,-300){$#2$}
      \put(2200,4000){$#3$}
      \put(1950,-300){$#4$}
      \put(3300,1800){$#5$}
      \put(500,1800){$#6$}
      \put(2250,1150){$#7$}      
    \end{picture}
}

\theoremstyle{plain}
\newtheorem{theorem}{Theorem}
\newtheorem{prop}[theorem]{Proposition}
\newtheorem{lemma}[theorem]{Lemma}

\theoremstyle{definition}
\newtheorem{exmp}[theorem]{Example}
\newtheorem{rem}[theorem]{Remark}

\begin{document}
\title{The full automorphism groups of the five symmetric $(15,8,4)$-designs}
\author{Mark Pankov, Krzysztof Petelczyc, Mariusz \.Zynel}
\address{Mark Pankov: Faculty of Mathematics and Computer Science, 
University of Warmia and Mazury, S{\l}oneczna 54, 10-710 Olsztyn, Poland}
\email{pankov@matman.uwm.edu.pl}
\address{Krzysztof Petelczyc, Mariusz \.Zynel: Faculty of Mathematics, University of Bia\-{\l}y\-stok, Cio{\l}kowskiego 1M, 15-245 Bia{\l}ystok, Poland}
\email{kryzpet@math.uwb.edu.pl, mariusz@math.uwb.edu.pl}

\begin{abstract}
It is clear that the full automorphism group of the $(15,8,4)$-design of points and hyperplane complements of $\PG(3,2)$ is $\GL(4,2)$. 
Using methods of point-line geometries, we determine the full automorphism groups of the remaining four symmetric $(15,8,4)$-designs
and describe their actions on the sets of points and blocks.
\end{abstract}

\maketitle

\section{Introduction}
It was shown in \cite{Nandi} that there are precisely five pairwise non-isomorphic symmetric $(15,7,3)$-designs 
(they are presented, for example, in \cite[p.11, Table 1.23]{CD}). The same holds for symmetric $(15,8,4)$-designs,
since they are the complements of $(15,7,3)$-designs. 
One of the $(15,7,3)$-designs is the design of points and hyperplanes of $\PG(3,2)$
and its full automorphism group is $\GL(4,2)$. 
Note that  its complement can be characterized as a  unique symmetric $(15,8,4)$-design admitting a flag-transitive, point-imprimitive subgroup of automorphisms
\cite[Proposition 1.5]{PZ} (previously, it was shown in \cite[Subsection 1.2.2]{OR} that there is an automorphism group acting flag-transitive, point-imprimitive on this design).
In this paper, we determine the full automorphism groups of 
the remaining four designs (it is clear that a designs and its complement have the same automorphism group)
and describe their actions on the sets of points and blocks.

Point-line geometries (more precisely, partial Steiner triple systems) related to binary equidistant codes are investigated in \cite{PPZ1}. 
Some of maximal cliques in the collinearity graphs of these geometries are the sets of blocks of symmetric designs.
For the geometry whose maximal singular subspaces correspond to 
binary simplex codes of dimension $k$ (recall that such codes are of length $2^k-1$)
a clique of the collinearity graph contains at most  $2^k-1$ vertices.
Furthermore, $(2^k-1)$-cliques are precisely the sets of blocks of  symmetric $(2^k-1,2^{k-1},2^{k-2})$-designs.  

The five symmetric $(15,8,4)$-designs (the case when $k=4$) admit a simple description in terms of points and lines \cite{PPZ2}.
Four of them are centered, i.e.\ the set of blocks is the set of all points on seven lines of $\PG(15,2)$ passing through a common point. 
The remaining non-centered design is dual to one of the centered designs.
Since a design and its dual have the same full automorphism group, we restrict ourself to centered designs only.

\section{Main result}
The set of points of $\PG(n-1,2)$ can be identified with the set of non-empty subsets of $[n]=\{1,\dots,n\}$ as follows:
every non-zero vector $x=(x_1,\dots,x_n)\in {\mathbb F}^n_2$ (where ${\mathbb F}_2$ is the two-element field)
corresponds to the subset of all $i\in [n]$ such that the $i$-th coordinate of $x$ is non-zero. 
Then the third point on the line of $\PG(n-1,2)$ joining $X,Y\subset [n]$ is the symmetric difference $X\triangle Y$.  

For every $m\in [n]$ satisfying $3m\le n$ denote by $\P_m(n)$ the subgeometry of $\PG(n-1,2)$
formed by all $2m$-element subsets of $[n]$ (points of Hamming weight $2m$);
lines of $\P_m(n)$ are precisely the lines of $\PG(n-1,2)$ contained in $\P_m(n)$.
Note that the set of $2m$-subsets of $[n]$ contains no line if $3m>n$;
the same holds for any subset of $\PG(n-1,2)$ formed by all points of fixed odd Hamming weight.

Two distinct points $X,Y$ ($2m$-element subsets of $[n]$) are collinear (connected by a line) in $\P_m(n)$
if and only if $X\triangle Y$ consists of $2m$ elements or, equivalently, $X\cap Y$ consists of $m$ elements.

Every permutation on $[n]$ induces an automorphism of $\P_m(n)$
(a bijective transformation of $\P_m(n)$ preserving the set of lines).
In the cases when $n=4m-1,4m$, there are exceptional automorphisms of $\P_m(n)$
which are not induced by permutations. See \cite{PPZ1} for a full description of automorphisms  of $\P_m(n)$.

A {\it singular subspace} of $\P_m(n)$  is a subset, where any two distinct points are collinear and the line joining them is contained in this subset. 
Every singular subspace is a projective space.

\begin{exmp}\label{exmp-SS}
Suppose that $n=2^k-1$ and $m=2^{k-2}$. Then every maximal singular subspace of $\P_m(n)$ is isomorphic to $\PG(k-1,2)$
and corresponds to a certain binary simplex code of dimension $k$, see \cite{PPZ1} for the details.
Every automorphism of this subspace is induced by a unique permutation on $[n]$.
This follows from the fact that every $i\in [n]$ is contained in precisely $2m$ points of the subspace 
and their intersection is the one-element set $\{i\}$.
\end{exmp}

The {\it collinearity graph} of $\P_m(n)$ is the simple graph whose vertices are points of $\P_m(n)$
and two distinct vertices are connected by an edge if they are collinear points. 
By Fisher's inequality (see, for example, \cite{MM}),
every clique of the collinearity graph contains at most $n$ vertices. 
It was observed in \cite{PPZ1} that every $n$-clique (a clique consisting of $n$ vertices) is the set of blocks of a certain symmetric 
$(n,2m,m)$-design. 

Recall that a {\it symmetric $(n,2m,m)$-design} is an incidence structure formed by $n$ points and $n$ blocks such that
each block consists of $2m$ points and the intersection of two distinct blocks contains precisely $m$ points.
{\it Design isomorphisms} are block preserving bijections between the point sets of designs 
and {\it design automorphisms} are isomorphisms of a design to itself. 
The full automorphism group of a design $\D$ is denoted by $\Aut(\D)$. 
A design and its dual 
(derived by swapping the roles of points and blocks) 
have the same full automorphism group.

It is well-known that there are precisely five mutually non-isomorphic symmetric $(15,8,4)$-designs
(their complements are listed, for example,  in \cite[p.11, Table 1.23]{CD}).
Following \cite{PPZ2} we describe the sets of blocks of these five designs as subsets of $\PG(15,2)$ which are $15$-cliques of 
the collinearity graph of $\P_{4}(15)$:
\begin{enumerate}
\item[(C1)] A maximal singular subspace of $\P_{4}(15)$
(it is isomorphic to $\PG(3,2)$). The corresponding design is isomorphic to the design of points and hyperplane complements of $\PG(3,2)$ \cite[Proposition 1]{PPZ2}.
\item[(C2)] Three Fano planes through a  certain line 
(a line of $\P_{4}(15)$ is contained in this subset if and only if it is a line in one of these Fano planes).
\item[(C3)] A Fano plane and four lines through a certain point 
(a line of $\P_{4}(15)$ is contained in this subset if and only if it is a line in the Fano plane or one of these four lines).
\item[(C4)] Seven lines through a point 
(a line of $\P_{4}(15)$ is contained in this subset if and only if it is one of these seven lines).
\item[(NC)]  The union of a maximal singular subspace of $\P_4(15)$ (which is isomorphic to $\PG(3,2)$) with a Fano plane deleted
and a Fano plane disjoint with this subspace 
(a line of $\P_{4}(15)$ is contained in this subset if and only if it is a line of the Fano plane).
\end{enumerate}
The first four cliques are called {\it centered}; each of them contains a {\it center point} (not necessarily unique), i.e.\ 
a point $O$ such that  for any different point belonging to the clique  the line joining this point and $O$ is contained in the clique. 
In a clique of the type (C1), each point is a center point. In a clique of the type (C2), center points form a line
(which is the intersection of the three Fano planes).
A clique of the type (C3) or (C4) contains a unique center point. 
Cliques of type (NC) are non-centered. 

We say that a symmetric $(15,8,4)$-design is of type ${\rm A}\in \{\rm{(C1),\dots, (C4), (NC)}\}$ if the blocks of this design 
form a clique of type ${\rm A}$. Since every permutation on $[n]$ induces an automorphism of the geometry $\P_m(n)$,
designs of different types are non-isomorphic. On the other hand, there are precisely five mutually non-isomorphic symmetric $(15,8,4)$-designs
and, consequently, two such designs are isomorphic if and only if they are of the same type. 

Designs of the types (C1), (C2), (C3) are self-dual which means that the incidence matrices of such designs are symmetric, see \cite[Fig.~5-7]{PPZ2}.
On the other hand, designs of the type (C4) are dual to designs of the type (NC) and vice versa, 
i.e.\ the incidence matrices for one of these types are the transposes of 
the incidence matrices for the other, see \cite[Fig.~8,9]{PPZ2}.

It is clear that the full automorphism group of a design of the type (C1) is $\GL(4,2)$.
We determine the full automorphism groups for designs of the remaining four types. 
Designs of the types (C4) and (NC) have the same full automorphism group (as dual)
and we restrict ourself to the cases (C2)-(C4).

\begin{theorem}\label{theorem-main}
The following assertions are fulfilled:
\begin{itemize}
\item
If\/ $\D$ is a design  of the type {\rm (C2)}, then $\Aut(\D)$
is the semidirect product of a normal subgroup isomorphic to $C^2_2\rtimes S_4$ and a subgroup isomorphic to $S_3$.
\item
If\/ $\D$ is a design  of the type {\rm (C3)}, then $\Aut(\D)$
is the semidirect product of a normal subgroup isomorphic to $C^3_2$ and 
a subgroup isomorphic to $A_4$.
\item 
If\/ $\D$ is a design of the type  {\rm (C4)} or {\rm (NC)}, then $\Aut(\D)$
is the semidirect product of a normal subgroup isomorphic to $C^3_2$ and 
a subgroup isomorphic to $C_7{\rtimes}C_3$.
\end{itemize}
\end{theorem}

The following interpretation of the group $C^3_2$ as a Fano plane plays an important role in our description.

\begin{rem}\label{remF}
Recall that $C^3_2$ is an abelian group generated by three involutions.  
It contains precisely $7$ non-identity elements and each of them is an involution. 
There is a one-to-one correspondence  between non-identity elements of $C^3_2$ and points of a Fano plane:
an involution $\alpha$ corresponds to a point $p_{\alpha}$ such that
for distinct involutions $\alpha,\beta$ the point $p_{\alpha\beta}$ is the third point on the line joining $p_{\alpha}$ and $p_{\beta}$.
Every maximal subgroup of $C^3_2$  is $C^2_2$ (generated by two involutions) and there is a one-to-one correspondence between 
such subgroups and lines of the Fano plane. 
We will use the dual presentation of $C^3_2$: 
every involution $\alpha$ corresponds to a line $L_{\alpha}$ such that for distinct involutions $\alpha,\beta$
the line $L_{\alpha\beta}$ is the third line through the common point of the lines $L_{\alpha},L_{\beta}$.
\end{rem}

Now, let us briefly describe how we intend to determine the full automorphism groups of our designs.

Designs of the types (C1)-(C4) can be obtained from  bijective transformations of Fano planes
and there is a one-to-one correspondence between the types of designs  and equivalence classes of such transformations. 
These equivalence classes are determined by the number of lines which go to lines under transformations.
Every bijective transformation $\delta$ of a Fano plane $\F$ defines a new Fano plane  $\F_{\delta}$ whose points are the points of $\F$
and whose lines are the images of lines under $\delta$. 
The group $G_{\delta}$ formed by all common automorphisms of $\F$ and $\F_{\delta}$ will be described in Section 3. 

Let $\D$ be a design of one of the types (C1)-(C4).
Then the set of blocks ${\mathcal B}$ is a centered $15$-clique of the collinearity graph of $\P_4(15)$. 
For every center point $O\in {\mathcal B}$ the group $\Aut(\D,O)$
consisting of all automorphisms of $\D$ preserving $O$ contains a normal subgroup isomorphic to $C^3_2$. 
This subgroup leaves every element of $[15]\setminus O$ fixed and acts sharply transitively on $O$.
The group $\Aut(\D,O)$ is the semidirect product of this normal subgroup and a subgroup isomorphic to $G_{\delta}$,
where $\delta$ is the Fano plane transformation related to $\D$ (Subsection 4.2).
If $\D$ is of the type (C3) or (C4), then $\Aut(\D,O)$ coincides with $\Aut(\D)$.

The case (C2) is considered in Subsection 4.3. 
The center points of ${\mathcal B}$ form a line ${\mathcal O}$ and 
the group $\Aut(\D,{\mathcal O})$ consisting of all automorphisms of $\D$ preserving each center point 
is a normal subgroup of $\Aut(\D)$. It is the semidirect product of a normal subgroup isomorphic to $C^2_2$
and a subgroup isomorphic to $G_{\delta}$
(as above, $\delta$ is the Fano plane transformation related to $\D$). 
The group $\Aut(\D)$ is the semidirect product of the normal subgroup $\Aut(\D,{\mathcal O})$
and a subgroup isomorphic to $S_3$.

The action of the full automorphism group of a design of the type (C1) is flag-transitive. 
The orbits of the action of the full automorphism group  for the remaining four designs
will be determined in Section 5.

\section{Subgroups of \texorpdfstring{$\GL(3,2)$}{GL(3,2)} related to bijections of Fano planes}

\subsection{Some technical results}
The description of centered $15$-cliques in the col\-linearity graph of $\P_4(15)$
is based on a classification of bijective maps of Fano planes.

Let $\delta$ be a bijection between Fano planes $\F$ and $\F'$.
The {\it index} $\ind(\delta)$ is the number of lines in $\F$ whose images under $\delta$ are lines in $\F'$.
By \cite[Subsection 4.2]{PPZ2}, the index is equal to $0,1,3$ or $7$ ($\delta$ is an isomorphism  in the last case);
furthermore, a bijection $\tilde\delta:\F\to \F'$ is equivalent to $\delta$
(there are automorphisms $f$ of $\F$ and $g$ of $\F'$ such that 
$\tilde\delta=g \delta f$) if and only if $\delta, \tilde\delta$ are of the same index.  

The preimage of every line in $\F'$ under $\delta$ is said to be a {\it$\delta$-line} in $\F$.
The points of $\F$ together with all $\delta$-lines form a Fano plane which will be denoted by $\F_{\delta}$. 
Then $\delta$ is an isomorphism between the Fano planes $\F_{\delta}$ and $\F'$.
We have $\F=\F_{\delta}$ if and only if $\ind(\delta)=7$ ($\delta$ is an isomorphism). 

The full automorphism group of  $\F$ is $\GL(3,2)$.
Denote by $G_{\delta}$ the group formed by  all automorphisms of  $\F$
which are also automorphisms of  $\F_{\delta}$. 

\begin{lemma}\label{lemma-d}
An automorphism $f$ of $\F$  belongs to $G_{\delta}$
if and only if $\delta f\delta^{-1}$ is an automorphism of $\F'$. 
\end{lemma}

\begin{proof}
If $f\in G_{\delta}$, then for every line $L\subset \F'$ it sends $\delta^{-1}(L)$ to $\delta^{-1}(L')$, where $L'$ is a certain line of $\F'$.
This means that $\delta f\delta^{-1}$ transfers $L$ to $L'$, i.e.\ it preserves the set of lines of $\F'$.

Conversely, if $\delta f\delta^{-1}$  is an automorphism of $\F'$, then  $\delta f\delta^{-1}(L)=L'$ is a line for every line $L\subset \F'$.
Therefore, $f$ sends $\delta^{-1}(L)$ to $\delta^{-1}(L')$, i.e.\ it preservers the set of $\delta$-lines.
\end{proof}

It is clear that $G_{\delta}$ coincides with  $\GL(3,2)$ if $\ind(\delta)=7$. 
We describe $G_{\delta}$ for the cases when $\ind(\delta)\in \{0,1,3\}$. 

\begin{prop}\label{prop-d}
The following assertions are fulfilled:
\begin{itemize}
\item If\/ $\ind(\delta)=0$, then $G_{\delta}$ is a maximal subgroup of $\GL(3,2)$ isomorphic to $C_7{\rtimes}C_3$.
\item If\/ $\ind(\delta)=1$, then $G_{\delta}$ is a subgroup of $\GL(3,2)$ isomorphic to $A_4$. 
\item If\/ $\ind(\delta)=3$, then $G_{\delta}$ is a maximal subgroup of $\GL(3,2)$ isomorphic to $S_4$.
\end{itemize}
\end{prop}

\subsection{The case \texorpdfstring{$\ind(\delta)=0$}{ind(δ)=0}}
In this case, every $\delta$-line is not a line. 
There are precisely three $\delta$-lines through every point and there is precisely one $\delta$-line through two distinct points
which implies that  every line intersects precisely six $\delta$-lines. 
Therefore, for every line $L$ there is a unique $\delta$-line disjoint with $L$;
this $\delta$-line  will be called {\it opposite} to $L$. 
Also, there is a unique point which is not on $L$ and not on the $\delta$-line opposite to $L$; this point is said to be  {\it opposite} to $L$. 

Let $L$ be a line formed by points $x,y,z$ and let $p$ be the point opposite to this line.
Denote by $x',y',z'$ the points on the lines joining $p$ with $x,y,z$ (respectively), see Fig.~\ref{subfig:0lines}.
Then $x',y',z'$ form the $\delta$-line opposite to $L$.

\begin{figure}[h!]
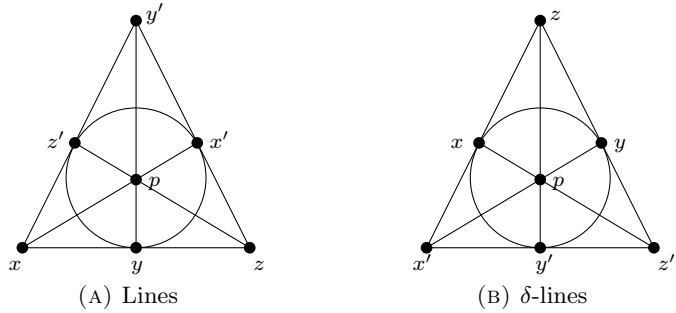

  \subfloat[Lines\label{subfig:0lines}]{\fano{x}{z}{y'}{y}{x'}{z'}{p}}
  \hfil
  \subfloat[$\delta$-lines\label{subfig:0deltalines}]{\fano{x'}{z'}{z}{y'}{y}{x}{p}}
  \caption{Lines and $\delta$-lines for $\ind(\delta)=0$}\label{fig:0case}
\end{figure} 

The $\delta$-line containing $x,y$ intersects the $\delta$-line $\{x',y',z'\}$ in a point. 
We assert that this intersecting point is distinct from $z'$. 
Indeed, if $\{x,y,z'\}$ is a $\delta$-line, then the $\delta$-line through $p,z'$ does not contain $x$ and as well as $y$ 
(the intersection of two distinct $\delta$-lines is a point); on the other hand, this $\delta$-line intersects $L$
(since it is not opposite to $L$) and, consequently, it contains $z$ which means that 
this $\delta$-line coincides with the line $\{p,z',z\}$, a contradiction.

The same arguments show that the $\delta$-lines through $x,z$ and $y,z$ intersect the $\delta$-line $\{x',y',z'\}$ in points distinct from $y'$ and $x'$, respectively.

Without loss of generality we assume that the $\delta$-line through $x,y$ is $\{x,y,y'\}$
(we have $G_{\delta f}=f^{-1}G_{\delta}f$ for every automorphism $f$ of $\F$ and take $f=(x,y)(x',y')$ if this $\delta$-line is $\{x,y,x'\}$).
If $\{x,z,z'\}$ is the $\delta$-line through $x,z$, then the $\delta$-line through $y,z$  is $\{y,z,x'\}$ which is impossible.
Therefore, 
$$\{x,y,y'\},\;\{y,z,z'\},\;\{x,z,x'\}$$
are the $\delta$-lines intersecting $L$ in two distinct points.

The $\delta$-line through $p,x'$ intersects $L$ in a point.
It is $\{p,x',y\}$; otherwise, this $\delta$-line coincides with the line $\{p,x',x\}$ or intersects the $\delta$-line
$\{x,z,x'\}$ in two points which is impossible. Using the same arguments, we determine the remaining two $\delta$-lines through $p$
and establish that 
$$\{p,x',y\},\;\{p,y',z\},\;\{p,z',x\}$$ 
are the $\delta$-lines through $p$. The seven $\delta$-lines are presented on  Fig.~\ref{subfig:0deltalines}.

Consider the automorphism
$$f=(x,y,z)(x',y',z')$$
of $\F$ which 
leaves $p$ fixed and permutes points of $L$ cyclically. It also
preserves the $\delta$-line $\{x',y',z'\}$ and induces the following two cycles on the set of $\delta$-lines:
$$(\{x,y,y'\},\{y,z,z'\},\{x,z,x'\})\qquad\text{and}\qquad  (\{p,x',y\},\{p,y',z\},\{p,z',x\}).$$ 
Thus, $f$ belongs to $G_{\delta}$.

Since $L$ is arbitrary chosen, 
for any line of $\F$  the group $G_{\delta}$ contains the automorphisms of $\F$ 
which leave fixed the point opposite to this line and permute points of the line cyclically.
For example, the line $\{p,x',x\}$ and the $\delta$-line $\{y,z,z'\}$  are disjoint which means that $y'$ is the point opposite to the  line $\{p,x',x\}$.
This means that
$$g=(x,p,x')(z',y,z)\quad\text{and}\quad gf=(x,z,p,x',y',y,z')$$
belong to $G_{\delta}$.
So, $G_{\delta}$ contains subgroups isomorphic to $C_3$ and $C_7$.

Note that $G_{\delta}$ is a proper subgroup of $\GL(3,2)$
(since $(y,z)(y',z')$ is an automorphism of $\F$ which is not an automorphism of $\F_{\delta}$, 
see Fig.~\ref{subfig:0lines},\ref{subfig:0deltalines}).
It is well-known that every maximal subgroup of $\GL(3,2)$ is isomorphic to $C_7{\rtimes}C_3$ or $S_4$ \cite[p.3]{Atlas}.
Also, the maximal subgroups of $C_7{\rtimes}C_3$ are precisely $C_3$ and $C_7$.
Therefore, $G_{\delta}$ coincides with a maximal subgroup of $\GL(3,2)$ isomorphic to $C_7{\rtimes}C_3$.

\subsection{The case \texorpdfstring{$\ind(\delta)=1$}{ind(δ)=1}}
In this case, there is precisely one line $L$ which is also a $\delta$-line. This line is invariant for $G_{\delta}$.
Denote by $S$ the complement of $L$.
All $\delta$-lines mentioned below are assumed to be distinct from $L$. 
For every point $x\in L$ precisely two of the six $\delta$-lines intersect $L$ in $x$.
If one of these two $\delta$-lines is $\{x\}\cup A$, where $A$ is a $2$-element subset of $S$,
then the second $\delta$-line is $\{x\}\cup (S\setminus A)$. 

Let $x,y,z$ be the points of $L$ and let $p$ be a point not on $L$. 
Denote by $x',y',z'$ the points on the lines joining $p$ with $x,y,z$, respectively.
Every permutation on $S=\{p,x',y',z'\}$ can be uniquely extended to an automorphism of $\F$.
Therefore, 
without loss of generality we assume that the six $\delta$-lines are 
$$\{x,z',p\}, \{x,x',y'\},\{y,x',p\}, \{y,y',z'\}, \{z, y',p\}, \{z, x',z'\},$$
\begin{figure}[ht!]
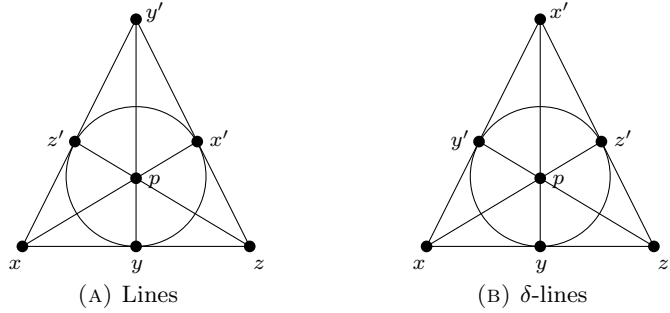

  \subfloat[Lines\label{subfig:1lines}]{\fano{x}{z}{y'}{y}{x'}{z'}{p}}
  \hfil
  \subfloat[$\delta$-lines\label{subfig:1deltalines}]{\fano{x}{z}{x'}{y}{z'}{y'}{p}}
  \caption{Lines and $\delta$-lines for $\ind(\delta)=1$}\label{fig:1case}
\end{figure}%
see Fig.~\ref{fig:1case}. A direct verification shows that 
$$g_x=(z',p)(x',y')\in G_{\delta}$$
preserves each 
$\delta$-line through $x$ and transposes 
$\delta$-lines  through $y$ and 
$\delta$-lines through $z$.
The elements
$$g_y=(x',p)(y',z')\qquad\text{and}\qquad g_z=(y',p)(x',z')\in G_{\delta}$$
act similarly. 
Since each of $g_x,g_y,g_z$ is the product of the remaining two, 
the group $\{e,g_x,g_y,g_z\}$ is isomorphic to $C^2_2$. 

If $g\in G_{\delta}$ leaves one of the points of $L$, say $x$,  fixed, 
then it  preserves each of the $\delta$-lines through $x$  or transposes them. 
In the first case, $g$ preserves each of  the subsets $\{z',p\},\{x',y'\}$ which means that it is identity or $g_x$ 
(it cannot transpose the elements in one of these subsets and, in the same time, leave the elements in the other one fixed). 
In the second case, $g$ transposes  $\{z',p\}$ and $\{x',y'\}$ which implies that $g$ is $g_y$ or $g_z$.

So, an element of $G_{\delta}$ leaves the points of $L$ fixed which guarantees that it belongs to $\{e,g_x,g_y,g_z\}$ or it permutes the points of $L$ cyclically.
For example,  $$(x,y,z)(x',y',z')\in G_{\delta}$$ satisfies the second condition
and generates a subgroup isomorphic to $C_3$.
Since $\{e,g_x,g_y,g_z\}$ is a normal subgroup of $G_{\delta}$ (as the kernel of the homomorphism sending every $f\in G_{\delta}$ to $f|_L$),
we obtain that $G_{\delta}$ is  isomorphic to $C^2_2{\rtimes} C_3=A_4$.

\subsection{The case \texorpdfstring{$\ind(\delta)=3$}{ind(δ)=3}}
In this case, there are precisely three lines which are also $\delta$-lines.
These are the lines passing through a certain  point $p$ (see \cite[Subsection 4.2]{PPZ2} for the description of bijections of index $3$).
Every element of $G_{\delta}$ leaves $p$ fixed. 

Let $x,y,z$ be the points on a  line not through $p$ 
Denote by $x',y',z'$ the points on the lines joining $p$ with $x,y,z$, respectively.
Consider the four $\delta$-lines which are not lines. Each of them does not contain $p$. 
For every point distinct from $p$ there are precisely two such $\delta$-lines passing through this point. 
Also, each of these $\delta$-lines intersects every line through $p$ precisely in one point
and its complement is the union of $p$ and a line. 
The latter means that
every such $\delta$-line can be transferred to $\{x',y',z'\}$ by an automorphism of $\F$ leaving $p$ fixed.
So, without loss of generality, we assume that one of the $\delta$-lines is $\{x',y',z'\}$.
Then the remaining three $\delta$-lines are 
$$\{x',y,z\}, \{x,y',z\}, \{x,y,z'\},$$
\begin{figure}[ht!]
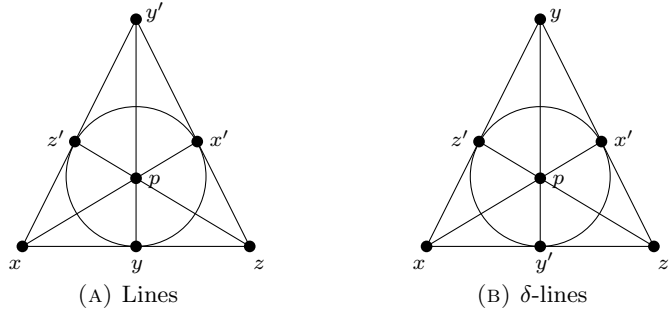

  \subfloat[Lines\label{subfig:3lines}]{\fano{x}{z}{y'}{y}{x'}{z'}{p}}
  \hfil
  \subfloat[$\delta$-lines\label{subfig:3deltalines}]{\fano{x}{z}{y}{y'}{x'}{z'}{p}}
  \caption{Lines and $\delta$-lines for $\ind(\delta)=3$}\label{fig:3case}
\end{figure}%
see Fig.~\ref{fig:3case}. 
Every permutation on the $\delta$-line $\{x',y',z'\}$ can be  extended to an automorphism of $\F$ as follows:
if $a'\in \{x',y',z'\}$ goes to $b'\in \{x',y',z'\}$, then we send $a\in \{x,y,z\}$ to $b\in \{x,y,z\}$ and leave $p$ fixed.
This automorphism preserves $\{x',y',z'\}$ and permutes the remaining three $\delta$-lines which are not lines.
All such automorphisms form a subgroup of $G_{\delta}$ isomorphic to $S_3$.
Also, $G_{\delta}$ contains $(x,y')(x',y)$.

It was noted above that $G_{\delta}$ is a subgroup of the stabilizer of $p$ in $\GL(3,2)$
which is a maximal subgroup of $\GL(3,2)$ isomorphic to $S_4$. 
Since $S_3$ is a maximal subgroup of $S_4$,
the group $G_{\delta}$ coincides with the stabilizer of $p$.

\section{Proof of Theorem \ref{theorem-main}}
\subsection{Structure of the set of blocks} 
Let ${\mathcal B}$ be a $15$-clique of the collinearity graph of $\P_4(15)$, i.e.\ 
the set of blocks of a symmetric $(15,8,4)$-design. 
We suppose that this clique is centered and $O$ is  a center point of ${\mathcal B}$, i.e.\ for every $B\in {\mathcal B}$ distinct from $O$
the line of $\P_4(15)$ joining $O$ and $B$ is contained in ${\mathcal B}$.
By \cite[Subsection 4.3]{PPZ2}, there is a Fano plane $\F_{O}$ formed by $4$-element subsets of $O^c=[15]\setminus O$
(note that $O^c$ contains precisely $7$ elements) and 
for every $7$-element subset $Z\subset O$ there is a Fano plane $\F_Z$ formed by $4$-element subsets of $Z$
such that ${\mathcal B}$ consists of $O$ and all 
$$X\cup \delta_Z(X),\; X\cup (O\setminus \delta_Z(X))\quad\text{with}\quad X\in \F_O,$$ 
where $\delta_Z$ is a certain bijection of $\F_O$ to $\F_Z$.
Note that $\F_O$ does not depend on $Z$. 
Also, for every $X\in \F_O$ the line joining $X\cup \delta_Z(X)$ and $X\cup (O\setminus \delta_Z(X))$ passes through $O$.

The line joining distinct $Y,Y'\in \F_Z$  consists of $Y,Y',Y\triangle Y'$ and each of these $4$-element subsets is contained in 
the $6$-element subset $Y\cup Y'$.
This provides a one-to-one correspondence between lines of $\F_Z$ and $6$-element subsets of $Z$.

For a $7$-element subset $Z'\subset O$ distinct from $Z$
the intersection $Z\cap Z'$ consists of $6$ elements. 
For every $Y\in \F_Z$ there is a unique $X\in \F_O$ such that $X\cup Y$ and $X\cup (O\setminus Y)$ belong to ${\mathcal B}$.
If $Y$ is on the line of $\F_Z$ corresponding to $Z\cap Z'$,
then $Z'$ contains $Y$ which means that $Y$ is also a point of $\F_{Z'}$. 
Therefore, the intersection of $\F_Z$ and $\F_{Z'}$ is the line corresponding to $Z\cap Z'$. 
If $Y$ is not on this line, then $Z'$ contains $O\setminus Y$
and, consequently,  $O\setminus Y$ is a point of $\F_{Z'}$.

\begin{rem}\label{remC2}
If $X\in \F_O$, then $X\cup \delta_{Z'}(X)$ is one of $X\cup \delta_Z(X)$, $X\cup (O\setminus \delta_Z(X))$.
Therefore, 
$$\delta_{Z'}(X)=\delta_Z(X)$$
if $\delta_Z(X)$  is on the common line $\F_Z\cap \F_{Z'}$ and
$$\delta_{Z'}(X)=O\setminus \delta_Z(X)$$
otherwise. 
\end{rem}

If the index of $\delta_Z$ is $7,3,1,0$,
then our clique is of the type (C1), (C2), (C3), (C4), respectively \cite[Subsection 4.3]{PPZ2}.
The index of $\delta_Z$ does not depend on $Z$;
in the cases (C1) and (C2) (when there is more then one center point), it does not depend on a center point.
If $\ind(\delta_Z)=3$, then there are precisely three lines of $\F_O$ whose images under $\delta_Z$ are lines of $\F_Z$;
each of these two triples of lines has a common point \cite[Section 4.2]{PPZ2}.

As in the previous section, $G_{\delta_Z}$ is the group of all automorphism of $\F_O$ preserving  the set of $\delta_Z$-lines. 
This is $\GL(3,2)$ or one of the subgroups from Proposition \ref{prop-d}.
Every element of $G_{\delta_Z}$ is induced by a unique permutation on $O^c$. 

Let $\D$ be the design  whose blocks form ${\mathcal B}$.
Every automorphism of $\D$ induces an automorphism of $\P_4(15)$ preserving ${\mathcal B}$ and, consequently, it preserves the set of center points of ${\mathcal B}$.
Let $\Aut(\D,O)$ be the group formed by all automorphisms of $\D$ preserving $O$. 
We have $\Aut(\D)=\Aut(\D,O)$ in the cases (C3) and (C4) (when there is only one center point).

Denote by $G(Z)$ the group of all automorphisms of $\D$ preserving $Z$.
This is a subgroup of $\Aut(\D,O)$
(since $O$ is the unique center point of ${\mathcal B}$ containing $Z$, every automorphism of $\D$ preserving $Z$ also preserves $O$). 

\begin{lemma}\label{lemmaO1}
The groups $G(Z)$ and $G_{\delta_Z}$ are isomorphic.
\end{lemma}

\begin{proof}
Every $g\in G(Z)$ preserves $O,O^c,Z$ and induces automorphisms $f$ and $f'$ of $\F_O$ and $\F_Z$, respectively. 
Since for every $X\in \F_O$ it sends $X\cup \delta_Z(X)$ to $g(X)\cup\delta_Z (g(X))$,
we obtain that 
$\delta_Z f\delta^{-1}_Z=f'$ and, consequently, $f\in G_{\delta_Z}$ by Lemma \ref{lemma-d}.

Conversely, if $f\in G_{\delta_Z}$, then Lemma \ref{lemma-d} shows that $f'=\delta_Z f\delta^{-1}_Z$ is an automorphism of $\F_Z$. 
The automorphisms $f$ and $f'$ are induced by permutations on $O^c$ and $Z$, respectively.
The union of these permutations gives an element of $G(Z)$.
\end{proof}

\subsection{Action of the group \texorpdfstring{$C^3_2$}{C-3-2}. Proof in the cases (C3), (C4)} 
For every line $L\subset \F_Z$ there is a $7$-element subset $Z'\subset O$
such that  $\F_Z$ intersects $\F_{Z'}$ precisely in  $L$.
Suppose that 
$$Z\cap Z'=\{a_1,a_2,b_1,b_2,c_1,c_2\}$$
and the elements of $Z$ and $Z'$ which do not belong to $Z\cap Z'$ are $d$ and $d'$, respectively.
We assume that 
$$Y_1=\{a_1,a_2,b_1,b_2\},\quad Y_2=\{b_1,b_2,c_1,c_2\},\quad Y_3=Y_1\triangle Y_2=\{a_1,a_2,c_1,c_2\}$$
form the line  $L=\F_Z\cap \F_{Z'}$. 
Every $Y\in \F_Z\setminus  \F_{Z'}$ contains $d$
and precisely one of $x_1,x_2$ for each $x\in \{a,b,c\}$
(since $Y$ intersects each $Y_i$ precisely in two elements)
and $O\setminus Y$ is a point of $\F_{Z'}$.
The permutation 
$$\alpha_L=(a_1,a_2)(b_1,b_2)(c_1,c_2)(d,d')$$
preserves each $Y_i$ and transposes every $Y\in \F_Z\setminus  \F_{Z'}$ and $O\setminus Y$.
Therefore, it sends $\F_Z$ to $\F_{Z'}$ and vice versa. 
If $X_i$, $i\in\{1,2,3\}$ is the preimage of $Y_i$ under $\delta_Z$,
then $\alpha_L$ preserves $O$ and all
$$X_i\cup Y_i,\quad X_i\cup (O\setminus Y_i).$$
For every $X\in \F_{O}$ distinct from $X_i$ it transposes 
$$X\cup \delta(X)\qquad\text{and}\qquad X\cup (O\setminus \delta(X)).$$
So, $\alpha_L$ is an automorphism of $\D$ preserving each line of $\P_4(15)$ contained in ${\mathcal B}$ and passing through $O$.
Furthermore, it preserves each point on three of these lines and  transposes the points distinct from $O$ on the remaining four lines.

\begin{lemma}
For any mutually distinct lines $L_1,L_2,L_3\subset \F_Z$
passing through a common point
$$\alpha_{L_1}\alpha_{L_2}=\alpha_{L_2}\alpha_{L_1}=\alpha_{L_3}.$$
\end{lemma}

\begin{proof}
A direct verification shows that
$$\alpha_{L_1}\alpha_{L_2}(Y)=\alpha_{L_2}\alpha_{L_1}(Y)=\alpha_{L_3}(Y)$$
for every $Y\in \F_Z$. 
This implies that $\alpha_{L_1}\alpha_{L_2}$, $\alpha_{L_2}\alpha_{L_1}$, $\alpha_{L_3}$
induce the same transformation of ${\mathcal B}$ and, consequently, they are equal. 
\end{proof}

So,  all automorphisms $\alpha_L$ defined by the lines of $\F_Z$ form a group isomorphic to $C^3_2$.
We denote this group by $C^3_2(Z)$.

\begin{prop}\label{prop-FG}
For any $7$-element subsets $Z,Z'\subset O$ we have $C^3_2(Z)=C^3_2(Z')$. 
This is a normal subgroup of $\Aut(\D,O)$.
\end{prop}

\begin{proof}
We claim that 
$$hC^3_2(Z)h^{-1}=C^3_2(Z')$$
for every $h\in \Aut(\D,O)$ sending $\F_Z$ to $\F_{Z'}$.

If $L$ is a line in $\F_Z$, then $h(L)$ is a line in $\F_{Z'}$.
A direct verification shows that 
$$\alpha_{h(L)}(Y)=h\alpha_{L}h^{-1}(Y)$$
for every $Y\in \F_{Z'}$.
Since $\alpha_L$ and $\alpha_{h(L)}$ leave each element of  $O^c$ fixed,
$\alpha_{h(L)}$ and $h\alpha_{L}h^{-1}$ induce the same transformation of ${\mathcal B}$
and, consequently, $\alpha_{h(L)}=h\alpha_{L}h^{-1}$. 
This gives the claim.

If $L$ is the intersection of $\F_Z$ and $\F_{Z'}$, then $\alpha_L$ sends $\F_Z$ to $\F_{Z'}$ and
$$\alpha_LC^3_2(Z)\alpha_L=C^3_2(Z')$$
which  implies that  $C^3_2(Z)=C^3_2(Z')$ (since $\alpha_L$ is an element of  $C^3_2(Z)$). 

An arbitrary $h\in \Aut(\D,O)$ sends $\F_Z$ to a certain $\F_{Z''}$.
Then
$$hC^3_2(Z)h^{-1}=C^3_2(Z'')=C^3_2(Z)$$
which completes our proof.
\end{proof}

Since $C^3_2(Z)$ does not depend on $Z$,
we will denote this group by $C^3_2(O)$. 

\begin{prop}\label{prop-autO}
The group $\Aut(\D,O)$ 
is the semidirect product of the normal subgroup $C^3_2(O)$ and the subgroup $G(Z)$ for any $7$-element subset $Z\subset O$.
\end{prop}

\begin{proof}
Since the intersection of $C^3_2(O)$ and $G(Z)$ is trivial, it suffices to show that 
$\Aut(\D,O)$ is spanned by these subgroups. 

If $h\in \Aut(\D,O)$ preserves $\F_Z$, then it belongs to $G(Z)$.
Consider the case when $h\in \Aut(\D,O)$ sends  $\F_Z$ to $\F_{Z'}$ and $Z'\ne Z$. 
If $L$ is the common line of $\F_Z$ and $\F_{Z'}$, 
then $\alpha_L$ transposes $\F_Z$ and $\F_{Z'}$.
Therefore, $\alpha_Lh$ preserves $\F_Z$ and, consequently, belongs to $G(Z)$.
So, $h$ is the composition of an element of $G(Z)$ and $\alpha_L$.
\end{proof}

Suppose that our design is of the type (C3) or (C4). 
Then $O$ is the unique center point in ${\mathcal B}$ and $\Aut(\D)=\Aut(\D,O)$.
Lemma \ref{lemmaO1} states that $G(Z)$ is isomorphic to $G_{\delta_Z}$ and, consequently, $G(Z)$ is isomorphic to $A_4$ or $C_7\rtimes C_3$
by Proposition \ref{prop-d}.
Then the statement of Theorem \ref{theorem-main} follows from Proposition \ref{prop-autO}.

\subsection{Proof in the case (C2)}
Suppose that $\ind(\delta_Z)=3$. Then $\delta_Z$ sends the three lines through a certain $X\in \F_O$
to the lines $L_1,L_2,L_3\subset \F_Z$ passing through $\delta(X)$. 
By Subsection 3.4, $G_{\delta_Z}$ is the stabilizer of $X$ in $\GL(3,2)$. 
This implies that every element of $G(Z)$ preserves $\delta(X)$ (see the proof of Lemma \ref{lemmaO1}).

By \cite[Subsection 4.3]{PPZ2}, ${\mathcal B}$ is the union of the Fano planes $\F_i$, $i\in\{1,2,3\}$ obtained from  the lines $L_i$ as follows.
For every $i\in \{1,2,3\}$ we choose $X_i, \tilde{X_i}\in \F_O$  such that $\delta(X_i),\delta(\tilde{X_i})$ are points on the line $L_i$ distinct from $\delta(X)$.
Then 
$$O, X'\cup\delta(X'), X'\cup(O\setminus\delta(X')),\quad\text{where}\quad X'\in \{X,X_i, \tilde{X_i}\},$$
form a Fano plane in $\P_4(15)$ denoted by $\F_i$,  see Fig.~\ref{fig:FanoC2}.
This follows easily  from the fact that 
$$X,X_i,\tilde{X_i}\qquad\text{and}\qquad \delta(X),\delta(X_i),\delta(\tilde{X_i})$$
form lines in $\F_O$ and $\F_Z$ (respectively) for every $i$. 
The line $\mathcal O$ consisting of 
$$O, O'=X\cup\delta(X), O''=X\cup(O\setminus\delta(X))$$ 
is a common line of the planes and
each point on this line is a center point of ${\mathcal B}$. 
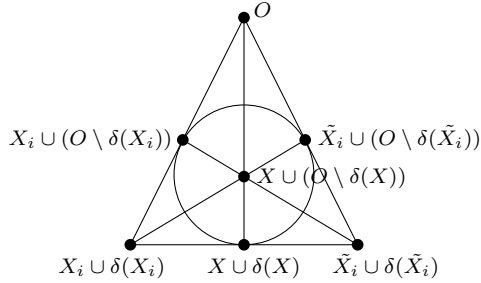
\begin{figure}[h!]
  \centering
  \footnotesize
  \setlength{\unitlength}{0.0003in}
  \begin{picture}(4066,4500)(0,-550)
    \put(2033,1283){\ellipse{2400}{2400}}
    \put(83,83){\circle*{200}}
    \put(2033,83){\circle*{200}}
    \put(3983,83){\circle*{200}}
    \put(983,1883){\circle*{200}}
    \put(2033,1253){\circle*{200}}
    \put(3083,1883){\circle*{200}}
    \put(2033,3983){\circle*{200}}
    \path(2033,3983)(3983,83)(83,83)(2033,3983)
    \path(3983,83)(983,1883)
    \path(2033,3983)(2033,83)
    \path(83,83)(3083,1883)
    \put(-1150,-400){$X_i\cup\delta(X_i)$}
    \put(3550,-400){$\tilde{X_i}\cup\delta(\tilde{X_i})$}
    \put(2200,4000){$O$}
    \put(1400,-400){$X\cup\delta(X)$}
    \put(3300,1800){$\tilde{X_i}\cup(O\setminus\delta(\tilde{X_i}))$}
    \put(-2000,1800){$X_i\cup(O\setminus\delta(X_i))$}
    \put(2250,1150){$X\cup(O\setminus\delta(X))$}      
  \end{picture}
  \caption{The Fano plane $\F_i$}\label{fig:FanoC2}
\end{figure}

Denote by $\Aut(\D,{\mathcal O})$ the group of all automorphisms of $\D$  preserving each point on the line ${\mathcal O}$.
Since every automorphism of $\D$ preserves this line (not necessarily pointwise), 
this is a normal subgroup of $\Aut(\D)$.

It is easy to see that $\Aut(\D,{\mathcal O})$ consists of all elements of $\Aut(\D,O)$ preserving $\delta(X)$.
In particular, $G(Z)$ is a subgroup of $\Aut(\D,{\mathcal O})$.

For the same reason $\alpha_L$ with $L\subset \F_Z$ belongs to $\Aut(\D,{\mathcal O})$ if and only if $\delta(X)$ is on the line $L$, i.e.\ $L$ is one of the lines $L_i$.
Therefore, 
\begin{equation}\label{eqCO}
C^3_2(O)\cap \Aut(\D,{\mathcal O})=\{e, \alpha_{L_1},\alpha_{L_2},\alpha_{L_3}\}.
\end{equation}
This group is isomorphic to $C^2_2$. 
Since $C^3_2(O)$ is a normal subgroup of $\Aut(\D,O)$,
\eqref{eqCO} is a normal subgroup of $\Aut(\D,{\mathcal O})$.

\begin{prop}\label{propC2}
The group $\Aut(\D,{\mathcal O})$ is the semidirect product of the normal subgroup \eqref{eqCO}
and the subgroup $G(Z)$.
\end{prop}

\begin{proof}
Since the intersection of the subgroups is trivial, we need to show that they generate $\Aut(\D,{\mathcal O})$.

In the case when $h\in \Aut(\D,{\mathcal O})$ preservers $\F_Z$,  it belongs to $G(Z)$. 
If $h\in \Aut(\D,{\mathcal O})$ sends $\F_Z$ to a certain $\F_{Z'}$ and $Z'\ne Z$,
then $\delta(X)$ is a common point of $\F_Z$ and $\F_{Z'}$
(since $h$ preservers $\delta(X)$).
This means that $\F_Z\cap \F_{Z'}$ is one of the lines $L_i$ 
and  the corresponding $\alpha_{L_i}$ transposes $\F_Z,\F_{Z'}$.
Then $\alpha_{L_i}h$ preserves $\F_Z$ and, consequently,  it belongs to $G(Z)$.
So, $h$ is the composition of elements from $G(Z)$ and \eqref{eqCO}.
\end{proof}

By Proposition \ref{prop-d} and Lemma \ref{lemmaO1}, $G(Z)$ is isomorphic to $S_4$.
Then Proposition  \ref{propC2} shows that $\Aut(\D,{\mathcal O})$ is isomorphic to $C^2_2\rtimes S_4$. 

Consider a line $L\subset \F_Z$ which does not pass through $\delta(X)$.  
This line intersects each $L_i$ in a point distinct from $\delta(X)$
and we can assume that $L$ consists of the points $\delta(X_i)$, $i\in\{1,2,3\}$ without loss of generality.
Then $\alpha_L$ acts on the Fano planes $\F_i$ as follows: 
it preserves each point on the line
$$\{O,\; X_i\cup\delta(X_i),\; X_i\cup(O\setminus\delta(X_i))\},$$ 
transposes the center points 
$$O'=X\cup \delta(X),\quad O''=X\cup (O\setminus \delta(X))$$ 
and the points  
$$\tilde{X}_i\cup\delta(\tilde{X}_i),\quad \tilde{X}_i\cup(O\setminus\delta(\tilde{X}_i)).$$ 
 Since $O'$ is also a center point of ${\mathcal B}$, we can repeat the above reasoning for any $7$-element subset $Z'\subset O'$  
and choose a line $L'\subset \F_{Z'}$ such that $\alpha_{L'}$ acts on each $\F_i$ as follows: 
it preserves all points on a certain line passing through $O'$, transposes the center points $O,O''$ and the remaining two points.
 
\begin{lemma}\label{lemmaS3}
The subgroup generated by $\alpha_L$ and $\alpha_{L'}$ is isomorphic to $S_3$. 
\end{lemma}

\begin{proof}
Since $\alpha_L,\alpha_{L'}$ are involutions, we need to show that $\alpha_L\alpha_{L'}$ is of order $3$.
Observe that $\alpha_L\alpha_{L'}$ permutes $O,O',O''$ cyclically and preserves a certain point on each $\F_i$.
This point is the intersection of the lines  fixed pointwise under $\alpha_L$ and $\alpha_{L'}$. 
Therefore,  $(\alpha_L\alpha_{L'})^3$ acts on ${\mathcal B}$ trivially and, consequently, it is the identity.
\end{proof}
 
 Since every automorphism of $\D$  induces a permutation on the line ${\mathcal O}$, 
the group $\Aut(\D)$ is the semidirect product of the normal subgroup $\Aut(\D,{\mathcal O})$ and the subgroup generated by $\alpha_L,\alpha_{L'}$.
It was established above that $\Aut(\D,{\mathcal O})$ is isomorphic to $C^2_2\rtimes S_4$ and the second subgroup is isomorphic to $S_3$.

\section{Action on the sets of points and blocks}
A design is called $t$-{\it pyramidal} when it has an automorphism group which fixes $t$
points and acts sharply transitively on the remaining points (see, for example, \cite{BRT}).
The following statement shows that designs of the types (C1)-(C4) are $7$-pyramidal. 

\begin{prop}\label{prop-pyr}
  If\/ $\D$ is a design of one of the types {\rm (C1)-(C4)} and $O$ is a center point in the clique of the collinearity graph of $\Gamma_4(15)$ formed by  the blocks of $\D$, then $\Aut(\D)$ contains a subgroup isomorphic to $C^3_2$ which leaves every element of $O^c$ fixed and acts sharply transitively on $O$.
\end{prop}

\begin{proof}
Consider the action of $C^3_2(O)$ described in Subsection 4.2. It leaves every element of $O^c$ fixed.
For any distinct $i,j\in O$ there is the unique pair of $7$-element subsets $Z,Z'\subset O$
such that $i\in Z\setminus Z'$ and $j\in Z'\setminus Z$. 
If $L$ is the common line of $\F_Z$ and $\F_{Z'}$, 
then $\alpha_L$ is the unique element of $C^3_2(O)$ transposing $i,j$.
\end{proof}

The action of the full automorphism group of a design of the type (C1) is flag-transitive. 
We determine the orbits of the action of the full automorphism group  for the remaining four designs. 

\begin{prop}
The following assertions are fulfilled:
\begin{enumerate}[\rm(1)]
\item If\/ $\D$ is a design of the type {\rm (C2)}, then each of the actions of $\Aut(\D)$
on the sets of points and blocks has precisely two orbits consisting of $3$ and $12$ elements. 
\item If\/ $\D$ is a design of the type {\rm (C3)}, then each of the actions of $\Aut(\D)$
on the sets of points and blocks has precisely one element fixed and two orbits consisting of $6$ and $8$ elements. 
\item If\/ $\D$ is a design of the type {\rm (C4)}, then the action of $\Aut(\D)$ on the set of points
has precisely two orbits consisting of\/ $7$ and $8$ elements;
also, $\Aut(\D)$ leaves one block fixed and acts transitively on the remaining blocks.  
\item If\/ $\D$ is a design of the type {\rm (NC)}, then  $\Aut(\D)$ leaves one point fixed and acts transitively on the remaining points;
the action of $\Aut(\D)$ on the set of blocks has precisely two orbits consisting of\/ $7$ and $8$ elements.
\end{enumerate}
\end{prop}

\begin{proof}
(1)\; If $O_1,O_2,O_3$ are the center points in the clique of the collinearity graph of $\Gamma_4(15)$ 
consisting of the blocks of $\D$, then they  form an orbit. 
Since $\D$ is self-dual, there is an orbit consisting of $3$ points of $\D$. 
The group $C^3_2(O_i)$ acts sharply transitively on the elements of $O_i$  for every $i\in \{1,2,3\}$
and there is an automorphism of $\D$ sending $O_i$  to $O_j$ for any distinct $i,j$. 
Therefore, the $12$ elements of $O_1\cup O_2\cup O_3$ form an orbit and its complement is the mentioned above orbit consisting of $3$ points of $\D$. 
The $12$ blocks different from  $O_1,O_2,O_3$ form an orbit by duality.

(2)\; If $O$ is the unique center point in the clique of the collinearity graph of $\Gamma_4(15)$  formed by the blocks of $\D$, 
then it is fixed under the action of $\Aut(\D)$ and $C^3_2(O)$ acts sharply transitively on the elements of $O$. 
Therefore, the elements of $O$ form an orbit and, since $\D$ is self-dual, there is a fixed point belonging to $O^c$.
Let $X$ be the complement of this point in $O^c$. 
The subgroup $G(Z)\subset \Aut(\D)$ (where $Z$ is a $7$-element subset of $O$) is  isomorphic to the group $G_{\delta}$ considered in Subsection 3.3.
The $6$-element subset  $X\subset O^c$ corresponds to a line in  the Fano plane $\F_O$ and this line is invariant under the action of $G(Z)$.
Let $A,B,C$ be the $2$-element subsets of $X$ such that $A\cup B, B\cup C, A\cup C$ are the points of this line. 
The subgroup $G(Z)$ contains an element $f$ of order $3$ permuting these points cyclically (see Subsection 3.3). 
Then $f$ permutes $A,B,C$ cyclically which implies that the action of $f$ on $X$ has precisely two $3$-element orbits 
and every such orbit  intersects each of $A,B,C$ precisely in one element. 
Also, $G(Z)$ contains an involution which preserves every point on the line of $\F_O$ corresponding to $X$. This involution transposes elements of at least one of $A,B,C$
which means that $X$ is an orbit of the action of $G(Z)$.
So, the action of $\Aut(\D)$ on the set of points has one fixed point and the $6$-element and $8$-element orbits $X$ and $O$.
The orbits in the set of blocks have the same numbers of elements  by duality. 

(3)\; In contrast to the previous cases, $\D$ is not self-dual. 
If $O$ is the unique center point in the clique of the collinearity graph of $\Gamma_4(15)$  formed by the blocks of $\D$, then, as in the case (2), 
it is fixed under the action of $\Aut(\D)$ and  the elements of $O$ form an orbit.
The subgroup $G(Z)\subset \Aut(\D)$ is  isomorphic to the group $G_{\delta}$ considered in Subsection 3.2.
There is $f\in G(Z)$ of order $7$ which permutes the points of $\F_O$ cyclically.
The restriction of $f$ to $O^c$ is a $7$-cycle which means that $O^c$ is the second orbit in the set of points.
The blocks distinct from $O$ form an orbit. 
Indeed, for any  $X,X'\in \F_O$ the block $X\cup \delta_Z(X)$ can be transformed to $X'\cup \delta_Z(X')$
by a certain power of $f$ and for any line $L\subset \F_Z$ which does not contain $\delta(X)$
the involution $\alpha_L$ transposes  $X\cup \delta_Z(X)$ and $X\cup (O\setminus\delta_Z(X))$.

(4)\;
The statement follows from (3) and the fact that designs of the types (C4) and (NC) are dual.
\end{proof}


\subsection*{Acknowledgements}


The authors are grateful to Alessandro Montinaro and  Cheryl Praeger for their interest in this research and useful comments.


\end{document}